\documentclass[11pt]{amsart}
\usepackage{amsmath,amssymb}
\usepackage{amscd,eucal,amsthm}
\newtheorem{Theorem}[subsection]{Theorem}
\newtheorem{Corollary}[subsection]{Corollary}
\newtheorem{Lemma}[subsection]{Lemma}

\newtheorem{Proposition}[subsection]{Proposition}

\theoremstyle{remark}
\newtheorem*{Remark}{Remark}

\newcommand{\ul}{\underline}
\renewcommand{\to}[1][]{\xrightarrow{#1}}

\newcommand{\C}{{\mathbb{C}}}

\renewcommand{\gg}{{\mathfrak{g}}}
\renewcommand{\ll}{{\mathfrak{l}}}
\newcommand{\fsl}{{\mathfrak{sl}}}

\newcommand{\fn}{{\mathfrak{n}}}
\newcommand{\fg}{{\mathfrak{g}}}
\newcommand{\fh}{{\mathfrak{h}}}
\newcommand{\fb}{{\mathfrak{b}}}

\renewcommand{\l}{{\lambda}}

\newcommand{\BZ}{\mathbb{Z}}
\newcommand{\BC}{\mathbb{C}}

\newcommand{\BR}{\mathbb{R}}
\newcommand{\rr}{\mathbb{R}}

\newcommand{\cc}{\mathbb{C}}

\newcommand{\PP}{\mathbb{P}}
\newcommand{\A}{\mathcal{A}} 
\DeclareMathOperator{\can}{can} \DeclareMathOperator{\Op}{Op}
\DeclareMathOperator{\diag}{diag} 
 \DeclareMathOperator{\gr}{gr}
\DeclareMathOperator{\ad}{ad} \DeclareMathOperator{\Ad}{Ad}
 \DeclareMathOperator{\Lie}{Lie}

\DeclareMathOperator{\Hom}{Hom} 
 \DeclareMathOperator{\End}{End}
 
 \DeclareMathOperator{\Spec}{Spec}
 \DeclareMathOperator{\rk}{rk}
 
\renewcommand{\phi}{\varphi}

\newcommand{\la}{\lambda}

 \makeatletter
\def\@mult#1{\raise #1\rlap{$\cdot$}\lower #1\rlap{$\cdot$}\cdot}
\def\did{\mathrel{\@mult{3pt}}}
\makeatother
\def\openrow#1#2#3{\setbox0=\vbox{\hbox
    {\vrule height#2 width#3\kern#2\vrule height#2 width0pt}\hrule height#3}
    \hbox{\leaders\copy0\hskip#1\wd0\vrule width#3}}
\def\row#1#2#3{\vbox{\hrule height#3\openrow{#1}{#2}{#3}}}
\def\Yr#1{\row{#1}{1.5ex}{.1ex}}

\def\DY#1\endDY{\baselineskip=1ex\lineskip=0pt\lineskiplimit=0pt{\vcenter
    {\Yr#1}}}
\def\openclm#1#2#3{\setbox0=\vbox{\hrule height#3\hbox
    {\vrule width0pt\kern#2\vrule width#3 height#2}}\vtop
    {\leaders\copy0\vskip#1\ht0\hrule height#3}}
\def\clm#1#2#3{\hbox{\vrule width#3\openclm{#1}{#2}{#3}}}
\def\Yc#1{\clm{#1}{1.5ex}{.1ex}}

\def\CDY#1\endCDY{{\vcenter{\hbox{\Yc#1}}}}

\newcommand{\nc}{\newcommand}
\nc{\CB}{{\mathcal{B}}}
\nc{\g}{{\mathfrak g}}
\nc{\LG}{{}^L\neg G}
\nc{\pone}{{\mathbb P}^1}
\nc{\wt}{\widetilde}
\nc{\wh}{\widehat}
\nc{\ghat}{\wh{\gg}}
\nc{\mc}{\mathcal}
\nc{\su}{{\mathfrak s}{\mathfrak l}}
\nc{\ppart}{(\!(t)\!)}
\nc{\on}{\operatorname}
\nc{\sw}{{\mf s}{\mf l}}
\nc{\mf}{\mathfrak}
\nc{\ol}{\overline}
\nc{\Gr}{\on{Gr}}
\nc{\bi}{\bibitem}

\begin{document}

\title{A proof of the Gaudin Bethe Ansatz conjecture.}
\author{Leonid Rybnikov}
\email{leo.rybnikov@gmail.com}
\address{National Research University Higher School of Economics,
Department of Mathematics,
International Laboratory of Representation Theory and Mathematical Physics,
and Institute for Information Transmission Problems,
20 Myasnitskaya st,
Moscow 101000, Russia}

\begin{abstract}
Gaudin algebra is the commutative subalgebra in $U(\g)^{\otimes N}$ generated by higher integrals of the quantum Gaudin magnet chain attached to a semisimple Lie algebra $\g$. This algebra depends on a collection of pairwise distinct complex numbers $z_1,\ldots,z_N$. We prove that this subalgebra has a cyclic vector in the space of singular vectors of the tensor product of any finite-dimensional irreducible $\g$-modules, for all values of the parameters $z_1,\ldots,z_N$. We deduce from this result the Bethe Ansatz conjecture in the Feigin-Frenkel form which states that the joint eigenvalues of the higher Gaudin Hamiltonians on the tensor product of irreducible finite-dimensional $\g$-modules are in 1-1 correspondence with monodromy-free ${}^LG$-opers on the projective line with regular singularities at the points $z_1,\ldots,z_N,\infty$ and the prescribed residues at the singular points.
\end{abstract}

\maketitle

\section{Introduction}
\subsection{Gaudin model.} The Gaudin model was introduced in \cite{G1} as a spin model related
to the Lie algebra $\mathfrak{sl}_2$, and generalized to the case of arbitrary
semisimple Lie algebras in \cite{G}, 13.2.2. The generalized Gaudin
model has the following algebraic interpretation. Let $V_{\lambda}$
be an irreducible representation of a semisimple (reductive) Lie algebra $\gg$ with the highest weight
$\lambda$. For any collection of integral dominant weights
$\ul{\lambda}=\lambda_1,\dots,\lambda_n$, let
$V_{\ul{\l}}=V_{\l_1}\otimes\dots\otimes V_{\l_N}$. For any $x\in\gg$,
consider the operator $x^{(i)}=1\otimes\dots\otimes 1\otimes
x\otimes 1\otimes\dots\otimes 1$ ($x$ stands at the $i$th place),
acting on the space $V_{\ul{\l}}$. Let $\{x_a\},\ a=1,\dots,\dim\gg$,
be an orthonormal basis of $\gg$ with respect to Killing form, and
let $\ul{z}:=(z_1,\dots,z_N)$ be a collection of pairwise distinct complex numbers. The
Hamiltonians of
the Gaudin model are the following commuting operators
acting in the space $V_{\ul{\l}}$:
\begin{equation}\label{quadratic}
H_i=\sum\limits_{k\neq i}\sum\limits_{a=1}^{\dim\gg}
\frac{x_a^{(i)}x_a^{(k)}}{z_i-z_k}.
\end{equation}

We can treat the $H_i$ as elements of the universal enveloping
algebra $U(\gg)^{\otimes N}$. In \cite{FFR}, the existence of  a
large commutative subalgebra $\A(\ul{z})\subset
U(\gg)^{\otimes N}$
containing $H_i$ was proved with the help of the critical level phenomenon for the affine Lie algebra $\ghat$.
For $\gg=\mathfrak{sl}_2$, the algebra
$\A(\ul{z})$ is generated by $H_i$ and the central elements
of $U(\gg)^{\otimes N}$. In other cases, the algebra
$\A(\ul{z})$ has also some new generators known as higher
Gaudin Hamiltonians. We will call $\A(\ul{z})$ the
\emph{Gaudin algebra}.

\subsection{Bethe Ansatz.} The main problem in Gaudin model is the problem of simultaneous
diagonalization of (higher) Gaudin Hamiltonians. It follows from the
\cite{FFR} construction that all elements of
$\A(\ul{z})\subset U(\gg)^{\otimes N}$ are invariant with
respect to the diagonal action of $\gg$, and therefore it is
sufficient to diagonalize the algebra $\A(\ul{z})$ in the
subspace $V_{\ul{\l}}^{sing}\subset V_{\ul{\l}}$ of singular (highest) vectors
with respect to $\diag(\gg)$ (i.e., with respect to the diagonal
action of $\gg$). In many important cases, the Gaudin eigenproblem is solved by the \emph{algebraic Bethe Ansatz} method which provides an explicit (but complicated) construction of joint eigenvectors for $\A(\ul{z})$ in $V_{\ul{\l}}^{sing}$ from the solutions of some explicit systems of algebraic equations, called Bethe Ansatz equations, on some auxiliary variables. The famous Bethe Ansatz conjecture states that this method always works, i.e. gives an eigenbasis for $\A(\ul{z})$ in $V_{\ul{\l}}^{sing}$.
In particular, the conjecture says that, for generic $\ul{z}$, the algebra $\A(\ul{z})$ has simple spectrum in
$V_{\ul{\l}}^{sing}$ and there is a 1-1 correspondence between the eigenvectors and the solutions of the corresponding system of Bethe Ansatz equations. The latter was proved by Mukhin, Tarasov and Varchenko in \cite{MTV07} for $\fg=\mathfrak{sl}_n$. In particular, it is proved that the space $V_{\ul{\l}}^{sing}$ is always cyclic as $\A(\ul{z})$-module, and hence $\A(\ul{z})$ has simple spectrum whenever acts by semisimple operators. Some partial results for other classical types are obtained by Lu, Mukhin and Varchenko in \cite{LMV}.

\subsection{Opers and the eigenproblem for the Gaudin model.} In \cite{FFR}, the Bethe Ansatz equations were interpreted as a ``no-monodromy'' condition on certain space of {\em opers} on the projective line $\PP^1$. Namely, it was shown that ${\mc A}(\ul{z})$ is isomorphic to
the algebra of functions on the space of $^L
G$-opers on $\PP^1$ with regular singularities at the points $z_1,\dots,z_N$ and
$\infty$. Here $^L G$ is the {\em Langlands dual group} of $G$ ($^L G$ is taken to be
of adjoint type), and $^L G$-opers are connections on a principal $^L
G$-bundle over $\pone$ satisfying a certain transversality condition,
as defined in \cite{BD}. The appearance of the Langlands dual group is
not accidental, but is closely related to the geometric Langlands
correspondence, through a description of the center of the completed
enveloping algebra of the affine Kac--Moody algebra $\ghat$ at the
critical level in terms of $^L G$-opers on the punctured disc
\cite{FF,Fr2,F:book}.

Thus, we obtain that the spectra of ${\mc A}(\ul{z})$ on a tensor product of $\g$-modules
$M_1\otimes\ldots\otimes M_N$ are encoded by $^L G$-opers on $\pone$ satisfying the above
properties. Furthermore, in \cite{FFR} it was shown that if each
$M_i$ is $V_{\la_i}$, the irreducible finite-dimensional $\g$-module with
dominant integral highest weight $\la_i$, then these $^L G$-opers
satisfy two additional properties: they have fixed residues at
the points $z_i$, determined
by $\la_i$ (we denote the space of such opers by $\Op_{^L G}(\PP^1)_{\ul{z}}^{\ul{\l}}$, it is an affine space), and they have {\em trivial monodromy}. The latter is a finite number of polynomial conditions on the coefficients of opers which is generically equivalent to Bethe Ansatz equations.

\subsection{Bethe ansatz conjecture.} The conjecture of completeness of Bethe Ansatz was reformulated in \cite{FFR} as follows: there is a bijection
between the joint eigenvalues of ${\mc A}(\ul{z})$ on $V_{\ul{\l}}$ and the set of
monodromy-free opers from $\Op_{^L G}(\PP^1)_{\ul{z}}^{\ul{\l}}$. In this paper we prove this conjecture. In fact, we prove the following statement:

\medskip

\noindent {\bf Main Theorem.} {\em The space of singular vectors in the tensor product of $\gg$-modules $V_{\ul{\l}}^{sing}$ is cyclic as
an $\A(\ul{z})$-module. The annihilator of $V_{\ul{\l}}^{sing}$ in $\A(\ul{z})$ is generated
by the no-monodromy conditions on opers from $\Op_{^L G}(\PP^1)_{\ul{z}}^{\ul{\l}}$.}

\medskip

It was shown in \cite{spectra} that for real values of the parameters $z_1,\ldots,z_N$, the algebra $\A(\ul{z})$ acts on $V_{\l_1}\otimes\ldots\otimes V_{\l_N}$ by Hermitian operators and hence is diagonalizable. Main Theorem then implies that the spectrum of $\A(\ul{z})$ on $V_{\ul{\l}}^{sing}$ is simple. Since the property of having simple spectrum is an open condition on the parameters $z_1,\ldots,z_N$ we have the following

\medskip

\noindent {\bf Main Corollary.} {\em For generic $z_1,\ldots,z_N$ and any dominant integral $\l_1,\ldots,\l_N$, the Gaudin subalgebra $\A(z_1,\ldots,z_N)\subset U(\gg)^{\otimes N}$ is diagonalizable and
has simple spectrum on space of singular vectors in the tensor product irreducible $\g$-modules $V_{\l_1}\otimes\ldots\otimes V_{\l_N}$.
Moreover, its joint eigenvalues (and hence eigenvectors, up to a
scalar) are in one-to-one correspondence with monodromy-free opers
from $\Op_{^L G}(\PP^1)_{\ul{z}}^{\ul{\l}}$.}

\medskip

The main idea of our proof is to use the \emph{inhomogeneous} Gaudin algebra ${\mc
A}_\mu(\ul{z})$, also known as Gaudin algebra with irregular singularities (see \cite{FFTL} and also \cite{MTV08} for different approach). This algebra depends on an additional parameter $\mu\in\fg=\fg^*$ and for regular $\mu$ is a maximal commutative subalgebra in $U(\fg)^{\otimes N}$.  On the one hand, the analog of Bethe Ansatz conjecture for inhomogeneous Gaudin algebras turns to be easier than for homogeneous ones. Namely, according to the results of \cite{spectra}, for any regular $\mu\in\fg$ and any collection of pairwise distinct complex numbers $\ul{z}=(z_1,\ldots,z_N)$ the algebra ${\mc
A}_\mu(\ul{z})$ has a cyclic vector in any tensor product of irreducible finite-dimensional $\fg$-modules, and the annihilator is generated by the no-monodromy conditions.
On the other hand, in the case when $\mu = f$, a regular {\em nilpotent} element
of $\gg^* \simeq \gg$, the algebra ${\mc A}_f(\ul{z})$ has a decreasing filtration such that the $0$-degree component of the associated (negatively) graded algebra is ${\mc A}(\ul{z})$, and it is possible to check that passing to the associated graded does not destroy the cyclicity property.

\subsection{Covering of $\ol{M_{0,N+1}}$ and relation to crystals.}

The family $\A(\ul{z})$, as defined, is parameterized by a noncompact complex algebraic variety of configurations of pairwise distinct points on the complex line. On the other hand, every subalgebra is (in appropriate sense) a point of some Grassmann variety which is compact. Hence there is a family of commutative subalgebras which extends the family $\A(\ul{z})$ and is parameterized by some compact variety. In \cite{cactus}, we show that the closure of the family $\A(\ul{z})$ is parameterized by the Deligne-Mumford compactification $\overline{M_{0,N+1}}$ of the moduli space of stable rational curves with $N+1$ marked points. Moreover, we show that the natural topological operad structure on $\overline{M_{0,N+1}}$ is compatible with that on commutative subalgebras of $U(\fg)^{\otimes N}$ and describe explicitly the algebras corresponding to boundary points of $\overline{M_{0,N+1}}$. In section~\ref{sect-cactus} we deduce from the Main Theorem that the subalgebras corresponding to boundary points of $\overline{M_{0,N+1}}$ have a cyclic vector in $V_{\ul{\l}}^{sing}$ as well. We deduce from this the simple spectrum property for the subalgebras attached to \emph{all real points} of $\overline{M_{0,N+1}}$.

This allows us to regard the eigenbasis (or, more precisely, the set of $1$-dimensional eigenspaces) of $\A(\ul{z})$ in $V_{\ul{\l}}^{sing}$ as a \emph{covering} of the space $\overline{M_{0,N+1}}(\BR)$. Denote the fiber of this covering at a point $\ul{z}\in\overline{M_{0,N+1}}(\BR)$ by $B_{\ul{\l}}(\ul{z})^{sing}$. The fundamental group of $\overline{M_{0,N+1}}(\BR)$ (called \emph{pure cactus group} $PJ_N$) acts on this set. On the other hand, there is an action of the same group on the set $\CB_{\ul{\l}}^{sing}$ of highest elements in the tensor product  $\CB_{\l_1}\otimes\ldots\otimes\CB_{\l_N}$ of the $\fg$-crystals with highest weights $\l_1,\ldots,\l_N$ from the coboundary category formalism. Note that $\CB_{\ul{\l}}^{sing}$ has the same cardinality as $B_{\ul{\l}}(\ul{z})^{sing}$. In \cite{cactus} we have stated the following conjecture suggested by Pavel Etingof:

\medskip

\noindent {\bf Monodromy Conjecture} (Etingof) \emph{The actions of $PJ_N$ on $B_{\ul{\l}}(\ul{z})^{sing}$ and on $\CB_{\ul{\l}}^{sing}$ are isomorphic.}

\medskip

The contribution of the present paper to this conjecture is that now the covering $B_{\ul{\l}}(\ul{z})^{sing}$ is well defined for any $\fg$ and $\ul{\l}$.

\subsection{The paper is organized as follows.} In
section~\ref{sect-prelim} we collect basic facts on Gaudin models and
opers. In section~\ref{sect-results} we prove the main
result of the paper. In the last section~\ref{sect-cactus} we prove the cyclicity and simple spectrum property for limit Gaudin algebras (corresponding to boundary points of $\ol{M_{0,N+1}}$).

\subsection{Acknowledgements.} This work is a part of the joint project with Joel Kamnitzer on Bethe algebras and crystals. I am grateful to Pavel Etingof and Joel Kamnitzer for useful discussions. I also thank Evgeny Mukhin for comments.

The article was prepared within the framework of the Academic Fund Program at the National Research University Higher School of Economics (HSE) in 2015- 2016 (grant №15-01-0062)  and supported within the framework of a subsidy granted to the HSE by the Government of the Russian Federation for the implementation of the Global Competitiveness Program.

\section{Preliminaries}\label{sect-prelim}

\subsection{Gaudin algebras} The Hamiltonians of the Gaudin model are the commuting quadratic elements $H_i\in U(\fg)^{\otimes N}$ defined by (\ref{quadratic}). In \cite{FFR} Feigin, Frenkel and Reshetikhin described a large commutative subalgebra $\A(\ul{z})$ which contains these Hamiltonians as well as some higher ones.  The construction of \cite{FFR} uses the affine Kac--Moody algebra $\ghat$ which is the universal central extension
of $\g\ppart$. Let us briefly describe it.

Consider the infinite-dimensional ind-nilpotent Lie
algebra $\fg_-:=\fg\otimes t^{-1}\BC[t^{-1}]$ -- it is a "half" of
the corresponding affine Kac--Moody algebra $\hat\fg$. The
universal enveloping algebra $U(\fg_-)$ has a natural (PBW) filtration
by the degree with respect to the generators. The associated
graded algebra is the symmetric algebra $S(\fg_-)$ by the
Poincar\'e--Birkhoff--Witt theorem.

There is a natural grading on the associative algebras $S(\fg_-)$ and $U(\fg_-)$ determined by the derivation $L_0$ defined by
\begin{equation}\label{der2}
L_0(g\otimes t^{m})=mg\otimes t^{m}\quad\forall g\in\fg,
m=-1,-2,\dots
\end{equation}

There is also a derivation $L_{-1}$ of degree $-1$ with respect to this grading:
\begin{equation}\label{der3}
L_{-1}(g\otimes t^{m})=mg\otimes t^{m-1}\quad\forall g\in\fg,
m=-1,-2,\dots
\end{equation}

The algebra of invariants,  $S(\fg)^{\fg}$, is known to be a free commutative algebra with $\rk\fg$ generators. Let $\Phi_j,\ j=1,\dots,\ell=\rk\fg$
be some set of free generators of the algebra $S(\fg)^{\fg}$. The degrees of $\Phi_j$ are $d_j+1$ where $d_j$ are the exponents of $\g$.

Let $i_{-1}:S(\fg)\hookrightarrow S(\fg_-)$ be the embedding,
which maps $g\in\fg$ to $g\otimes t^{-1}$. The following result is due to Boris Feigin and Edward Frenkel, see \cite{Fr2} and references therein.

\begin{Theorem}\label{thm-feigin-frenkel}
There exist commuting elements $S_j\in U(\fg_-)$, homogeneous with respect to $L_0$, such that $\gr S_j=i_{-1}(\Phi_j)$. Moreover, the elements $L_{-1}^kS_j$ pairwise commute for all $k\in\BZ_+$ and $j=1,\dots,\ell$.
\end{Theorem}

Let $U(\fg)^{\otimes N}$ be the tensor product of $N$ copies of
$U(\fg)$. We denote the subspace $1\otimes\dots\otimes
1\otimes\fg\otimes 1\otimes\dots\otimes 1\subset U(\fg)^{\otimes
N}$, where $\fg$ stands at the $i$th place, by $\fg^{(i)}$.
Respectively, for any $x\in U(\fg)$ we set
\begin{equation}
x^{(i)}=1\otimes\dots\otimes 1\otimes x\otimes
1\otimes\dots\otimes 1\in U(\fg)^{\otimes N}.
\end{equation}

Let $\diag:U(\fg_-)\hookrightarrow U(\fg_-)^{\otimes N}$ be the
diagonal embedding (i.e. for $x\in\fg_-$, we have $\diag(x)=\sum\limits_{i=1}^nx^{(i)}$). To any nonzero $w\in\BC$, we assign the homomorphism $\phi_w: U(\fg_-)\to U(\fg)$ of evaluation at the point $w$ (i.e., for $g\in\fg$, we have $\phi_w(g\otimes
t^m)=w^mg$). For any collection of pairwise distinct nonzero
complex numbers $w_i, i=1,\dots,n$, we have the following
homomorphism:
\begin{equation}
\phi_{w_1,\dots,w_N}=(\phi_{w_1}\otimes\dots\otimes\phi_{w_N})\circ
\diag:U(\fg_-)\to U(\fg)^{\otimes N}.
\end{equation}
More explicitly, we have
$$\phi_{w_1,\dots,w_N}(g\otimes
t^m)=\sum\limits_{i=1}^nw_i^mg^{(i)}.$$

Consider the following $U(\fg)^{\otimes N}$-valued functions in
the variable $w$
$$
S_j(w;z_1,\dots,z_N):=\phi_{w-z_1,\dots,w-z_N}(S_j).
$$

We define the  Gaudin subalgebra $\A(\ul{z})\subset U(\fg)^{\otimes N}$ as a subalgebra generated by $S_j(w;z_1,\dots,z_N)$ for all $w\in\BC\backslash\{z_1,\ldots,z_N\}$. Due to Theorem~\ref{thm-feigin-frenkel}, this subalgebra is commutative. The subalgebra $\A(\ul{z})\subset U(\fg)^{\otimes N}$ is also known as \emph{Bethe algebra}.

Let $S_j^{i,m}(z_1,\dots,z_N)$ be the coefficients of the principal
part of the Laurent series of $S_j(w;z_1,\dots,z_N)$ at the point
$z_i$, i.e.,
$$S_j(w;z_1,\dots,z_N)=\sum\limits_{m=1}^{m=\deg\Phi_j}S_j^{i,m}(z_1,\dots,z_N)
(w-z_i)^{-m}+O(1)\
\text{as}\ w\to z_i.$$

Taking the generator $S_j$ corresponding to the quadratic Casimir element on $S(\fg)$, one gets the quadratic Gaudin Hamiltonians (\ref{quadratic}) as the residues of $S_j(w;z_1,\dots,z_N)$ at the points $z_1,\ldots,z_N$. The following result is well-known (see e.g. \cite{CFR} for the proof).

\begin{Proposition}\label{generators2}\cite{CFR}
\begin{enumerate}
 \item The elements
$S_j^{i,m}(z_1,\dots,z_N)\in U(\fg)^{\otimes N}$ are semiinvariant under
simultaneous affine transformations of the parameters $z_i\mapsto
az_i+b$ (i.e. $S_j^{i,m}(az_1+b,\dots,az_N+b)$ is proportional to $S_j^{i,m}(z_1,\dots,z_N)$). \item The subalgebra $\A(\ul{z})$ is a free commutative
algebra generated by the elements $S_j^{i,m}(z_1,\dots,z_N)\in
U(\fg)^{\otimes N}$, where $i=1,\dots,N-1$, $l=1,\dots,\rk\fg$,
$m=1,\dots,\deg\Phi_j$, and by
$S_j^{N,\deg\Phi_j}(z_1,\dots,z_N)\in U(\fg)^{\otimes
n}$, where $l=1,\dots,\rk\fg$. \item All the elements of $\A(\ul{z})$ are
invariant with respect to the diagonal action of $\fg$. \item The
center of the diagonal $\diag(U(\fg))\subset U(\fg)^{\otimes N}$
is contained in $\A(\ul{z})$.
\end{enumerate}
\end{Proposition}

It is easy to see that one can replace $S_j^{N,d_j+1}(z_1,\dots,z_n)$ in (2) by the generators of the center of $\diag(U(\fg))$.

\subsection{Inhomogeneous Gaudin algebras.} In \cite{Ryb1} and \cite{FFTL}, the construction of \cite{FFR} was generalized. Namely, for any collection $z_1,\dots,z_N$ and $\mu\in\gg^*$, there is a commutative subalgebra
$\A_{\mu}(\ul{z})$ in $U(\gg)^{\otimes
n}$ which is invariant with respect to the diagonal action of the centralizer of $\mu$ in $\g$ and is maximal with this property. In particular, $\A(\ul{z}) =
\A_{0}(\ul{z})$ corresponding to $\mu=0$. Let us describe the generators of this algebra analogously to Proposition~\ref{generators2}.

To any $\mu\in\g^*$, we assign the character $\psi_\mu: U(\fg_-)\to \cc$ such that for $g\in\fg$, we have $\psi_\mu(g\otimes
t^m)=\delta_{-1,m}\mu(g)$). For any collection of pairwise distinct nonzero
complex numbers $w_i, i=1,\dots,N$, we can twist the
homomorphism $\phi_{w_1,\dots,w_N}$ by $\mu$:
\begin{multline}
\phi_{w_1,\dots,w_N;\mu}:=(\phi_{w_1}\otimes\dots\otimes\phi_{w_N}\otimes\psi_\mu)\circ
\diag:U(\fg_-)\to\\ \to U(\fg_-)^{\otimes (N+1)}\to U(\fg)^{\otimes N}\otimes\cc=U(\fg)^{\otimes N}.
\end{multline}
More explicitly, we have
$$\phi_{w_1,\dots,w_N;\mu}(g\otimes
t^m)=\delta_{-1,m}\mu(g)+\sum\limits_{i=1}^Nw_i^mg^{(i)}.$$

Consider the following $U(\fg)^{\otimes N}$-valued functions in
the variable $w$
$$
S_j(w;z_1,\dots,z_N;\mu):=\phi_{w-z_1,\dots,w-z_N;\mu}(S_j).
$$

We define the non-homogeneous Gaudin subalgebra $\A_\mu(\ul{z})\subset U(\fg)^{\otimes N}$ as a commutative subalgebra generated by $S_j(w;z_1,\dots,z_N;\mu)$ for all $w\in\BC\backslash\{z_1,\ldots,z_N\}$.

Let $S_j^{i,m}(z_1,\dots,z_N;\mu)$ be the coefficients of the principal
part of the Laurent series of $S_j(w;z_1,\dots,z_N;\mu)$ at the point
$z_i$, i.e.,
$$S_j(w;z_1,\dots,z_N;\mu)=\sum\limits_{m=1}^{m=\deg\Phi_j}S_j^{i,m}(z_1,\dots,z_N;\mu)
(w-z_i)^{-m}+O(1)\
\text{as}\ w\to z_i.$$

Taking the generator $S_j$ corresponding to the quadratic Casimir element on $S(\fg)$, one gets the non-homogeneous quadratic Gaudin Hamiltonians  as the residues of $S_j(w;z_1,\dots,z_N)$ at the points $z_1,\ldots,z_N$.
$$
H_i=\sum\limits_{k\neq i}\sum\limits_{a=1}^{\dim\gg}
\frac{x_a^{(i)}x_a^{(k)}}{z_i-z_k}+
\sum\limits_{a=1}^{\dim\gg}\mu(x_a)x_a^{(i)}.
$$

From \cite{spectra} we have the following description of the generators.

\begin{Proposition}\label{generators-mu}
\begin{enumerate}
\item The elements
$S_j^{i,m}(z_1,\dots,z_N;\mu)\in U(\fg)^{\otimes N}$ are homogeneous under
simultaneous affine transformations of the parameters $z_i\mapsto
az_i+b$ (i.e. $S_j^{i,m}(az_1+b,\dots,az_N+b,a^{-1}\mu)$ is proportional to $S_j^{i,m}(z_1,\dots,z_N;\mu)$). \item The subalgebra $\A_\mu(\ul{z})$ is a free commutative
algebra generated by the elements $S_j^{i,m}(z_1,\dots,z_N;\mu)\in
U(\fg)^{\otimes N}$, where $i=1,\dots,n$, $j=1,\dots,\ell$,
$m=1,\dots,d_j+1$.
\end{enumerate}
\end{Proposition}

Denote by $S_j^{\infty,m}(z_1,\dots,z_N;\mu)\in U(\fg)^{\otimes N}$ the coefficients of the principal
part of the Laurent series of $S_j(w;z_1,\dots,z_N;\mu)w^{d_j+1}$ at the point $\infty$.

\begin{Corollary} One can take as free generators of $\A_\mu(\ul{z})$ the elements $S_j^{i,m}(z_1,\dots,z_N;\mu)$ with $i=1,\dots,n-1$, $j=1,\dots,\ell$,
$m=1,\dots,d_j+1$ AND the elements $S_j^{\infty,m}(z_1,\dots,z_N;\mu)$ with $j=1,\dots,\ell$,
$m=1,\dots,d_j+1$.
\end{Corollary}

The following fact on the limit points of the family
$\A_{\mu}(\ul{z})$ is well known. We reproduce the proof here for completeness.

\begin{Proposition}\label{prop-lim} We have
$$\lim\limits_{s\to0}\A_{\mu}(sz_1,\dots,sz_N)=
\lim\limits_{s\to0}\A_{s\mu}(\ul{z})=
\diag(\A_{\mu})\cdot\A(\ul{z}) \subset U(\gg)^{\otimes
N}$$ for any regular $\mu\in\gg$. More precisely, both $\diag(\A_{\mu})$ and $\A(\ul{z})$ contain the center $Z$ of the diagonal $U(\fg)$ in $U(\fg)^{\otimes N}$, and the product $\diag(\A_{\mu})\cdot\A(\ul{z})$ is in fact the tensor product $\diag(\A_{\mu})\otimes_Z\A(\ul{z})$.
\end{Proposition}

\begin{proof} Indeed, the RHS has the same number of algebraically independent generators of the same degrees as that of $\A_{\mu}(z_1,\dots,z_N)$. So we just have to check that the limit subalgebra contains both $\diag(\A_{\mu})$ and $\A(\ul{z})$. It is easy to see that the first is generated by the limits of $S_j^{\infty,m}(z_1,\dots,z_N;\mu)$ and the second is generated by the limits of $S_j^{i,m}(z_1,\dots,z_N;\mu)$.
\end{proof}

\subsection{Opers on the projective line} Consider the Langlands dual Lie algebra ${}^L \fg$ whose Cartan
matrix is the transpose of the Cartan matrix of $\fg$. By $^L G$ we
denote the group of inner automorphisms of $^L \fg$. We fix a Cartan
decomposition
$$
^L \fg={}^L \fn_+\oplus {}^L \fh\oplus {}^L \fn_-.
$$
The Cartan subalgebra ${}^L \fh$ is naturally identified with
$\fh^*$. We denote by ${}^L \Delta$, ${}^L \Delta_+$, and ${}^L \Pi$
the root system of ${}^L \fg$, the set of positive roots, and the set
of simple roots, respectively. We denote by ${}^L \fb_+$ the Borel
subalgebra ${}^L \fn_+\oplus {}^L \fh$.

Set
$$
p_{-1} = \sum_{\alpha^\vee\in{}^L\Pi} e_{-\alpha^\vee}\in {}^L
\fg.
$$
Let
$$
\rho=\frac{1}{2}\sum_{\alpha\in\Delta_+}\alpha\in\fh^*={}^L
\fh.
$$
The operator $\Ad \rho$ defines the principal gradation on ${}^L
\fg$, with respect to which we have a direct sum decomposition
${}^L \fb_+ = \bigoplus_{i\geq 0} ^L \fb_i$.  Let $p_1$ be the
unique element of degree 1 in ${}^L \fn_+$ such that $\{
p_{-1},2\rho,p_1 \}$ is an $\fsl_2$-triple (note that $p_{-1}$
has degree $-1$). Let
$$
V_{\can} = \bigoplus_{i=1}^\ell V_{\can,i}
$$
be the space of $\Ad p_1$-invariants in ${}^L \fn_+$, decomposed
according to the principal gradation. Here $V_{\can,i}$ has degree
$d_i$, the $i$th exponent of $^L \fg$ (and of $\fg$). In particular,
$V_{\can,1}$ is spanned by $p_1$. Now choose a linear generator $p_j$
of $V_{\can,j}$.

Consider the Kostant slice in ${}^L \fg$, $$^L
\fg_{\can} = \left\{ p_{-1} + \sum_{j=1}^\ell y_j p_j, \quad y_j \in
\C \right\}.
$$
By \cite{Ko}, the adjoint orbit of any regular element in the Lie
algebra $ ^L \fg$ contains a unique element which belongs to $^L
\fg_{\can}$. Thus, we have canonical isomorphisms
\begin{equation}    \label{iso can}
{}^L \fg_{\can} \tilde\to {}^L \fg/^L G = {}^L\fh/{}^LW =  \fh^*/W =
\fg^*/G.
\end{equation}

Denote by ${}^L\pi$ the factorization map $^L\fg\to {}^L \fg/{}^L G={}^L \fg_{\can}$. For an affine curve $U = \Spec R$
and an \'etale coordinate $t$ on $U$, the space $\Op_{^L G}(U)$ of $^L
G$-opers on $U$ is isomorphic to the quotient of the space of
connections of the form
$$
d + (p_{-1} + {\mathbf v}(t))dt, \qquad {\mathbf v}(t) \in
{}^L\fb_+\otimes R
$$
by the free action of the group $^L N_+ \otimes R$ of regular algebraic maps
$U\to {}^L N_+$, where $^L N_+$ is the maximal unipotent subgroup of
$^L G$ such that $\Lie {}^L N_+={}^L \fn_+$. Each oper has a unique
representative of the form
$$
d + \left( p_{-1} + \sum_{j=1}^\ell v_j(t) p_j
\right)dt,\qquad v_j(t)\in R.
$$

In particular, the space $\Op_{^L G}(D^\times)$ of $^L G$-opers on the
formal punctured disc $D^\times=\Spec\C\ppart$ is isomorphic to the
space of connections of the form
$$
d + \left( p_{-1} + \sum_{j=1}^\ell v_j(t) p_j
\right)dt,\qquad v_j(t)\in\C\ppart.
$$

We will consider opers on $\PP^1$ with regular singularities at a finite number
of points. The oper has regular singularity at $z\in\PP^1$ if
has the following form in local coordinate $t$ at $z$:
$$
d + \left( p_{-1} +
\sum_{j=1}^\ell (v_{j} (t-z)^{-(d_j+1)}+o((t-z)^{-(d_j+1)})) p_j
\right)dt.
$$
The residue of such oper at the point $z$ is ${}^L\pi(p_{-1} - \rho +
\sum_{j=1}^\ell v_{j} p_j)\in {}^L \fg/{}^L G$.

We denote by $\Op_{^L
G}({\ul{z}})$ the space of $^L G$-opers on
$\PP^1\backslash\{z_1,\dots z_N,\infty\}$ with regular
singularities at the points $z_i, i=1,\ldots,N$, and at $\infty$. Each oper from this space may be
uniquely represented in the following form:
$$
d + \left( p_{-1} + \sum_{i=1}^N
\sum_{j=1}^\ell \sum_{n=0}^{d_j} v_{j,n}^{(i)} (t-z_i)^{-n-1} p_j
\right)dt,
$$
where
\begin{equation}\label{eq:reg-sing-oper}
\sum_{i=1}^N\sum_{n=0}^m v_{j,n}^{(i)}\binom{m}{n}z_i^{m-n} = 0
\end{equation}
for any $j=1,\ldots,\ell$ and $m=0,\ldots,d_j-1$. Thus, $\Op_{^L
G}({\ul{z}})$ is an affine space of dimension
$\frac{1}{2}(\dim\fg+\rk\fg)(N-1)+\rk\fg$.

Following \cite{FFTL} and \cite{spectra}, we denote by $\Op_{^L
G}(\PP^1)_{\ul{z};\pi(-\mu)}$ the space of $^L G$-opers on
$\PP^1\backslash\{z_1,\dots z_N,\infty\}$ with regular
singularities at the points $z_i, i=1,\ldots,N$, and with
irregular singularity of order $2$ at the point $\infty$ with the
$2$-residue $\pi(-\mu) \in\ ^L\g/^L G = \g^*/G$, where $\pi: \g^*
\to \g^*/G$ is the projection. Each oper from this space may be
uniquely represented in the following form:
$$
d + \left( p_{-1} + \sum_{j=1}^\ell \overline\mu_j p_j + \sum_{i=1}^N
\sum_{j=1}^\ell \sum_{n=0}^{d_j} v_{j,n}^{(i)} (t-z_i)^{-n-1} p_j
\right)dt,
$$
where
$$
p_{-1}+\sum_{j=1}^\ell \overline{\mu}_j p_j
$$
is the (unique) element of $^L \gg_{\can}$ contained in the $^L
G$-orbit corresponding to the $G$-orbit of $\mu \in \g^*$ under the
isomorphism \eqref{iso can}. Thus, $\Op_{^L
G}(\PP^1)_{\ul{z};\pi(-\mu)}$ is an affine space of dimension
$\frac{1}{2}(\dim\gg+\rk\gg)N$.

We will be interested in the special case of $\Op_{^L
G}(\PP^1)_{\ul{z};\pi(-\mu)}$ when $\mu$ is the principal nilpotent element $f\in\fg$. Since $\pi(f)=0$, the space $\Op_{^L
G}(\PP^1)_{\ul{z};\pi(-f)}$ is the space of opers with singularities as above but with \emph{zero} $2$-residue at $\infty\in\pone$. Hence the space $\Op_{^L
G}(\PP^1)_{\ul{z}}$ is naturally a subspace in $\Op_{^L
G}(\PP^1)_{\ul{z};\pi(-f)}$. On the punctured disc $D_\infty^\times$
at $\infty$ (with the coordinate $s=t^{-1}$) each element of $\Op_{^L
G}(\PP^1)_{\ul{z};\pi(-f)}$ may be represented by a connection of
the form (see \cite[Section~5.4]{FFTL}):
$$
d - \left(p_{-1} - \sum_{j=1}^\ell s^{-2d_j-1} u_j(s) p_j\right)ds, \qquad u_j(s) = \sum_{n=0}^{\infty}
u_{j,n} s^n.
$$
The equations~(\ref{eq:reg-sing-oper}) which define subspace $\Op_{^L
G}(\PP^1)_{\ul{z}}$ inside $\Op_{^L
G}(\PP^1)_{\ul{z};\pi(-f)}$ can be rewritten as $u_{j,n}=0$ for $j=1\ldots,\ell,\ n=0,\ldots,d_j-1$.

\subsection{Gaudin algebra and opers}
We have the following relation between Gaudin algebras and spaces of opers, see \cite{FFR,FFTL,spectra}.

\begin{Proposition}\label{multi-point-isomorphism} There is an isomorphism $\A(\ul{z})\simeq\cc[\Op_{^L G}(\PP^1)_{\ul{z}}]$. For regular $\mu$ there is an isomorphism $\A_\mu(\ul{z})\simeq\cc[\Op_{^L G}(\PP^1)_{\ul{z};\pi(-\mu)}]$. The generators $\Phi_1,\ldots,\Phi_\ell$ of $S(\g)^\g$ can be chosen such that these isomorphisms take $S_j(t;z_1,\dots,z_N)$ (resp. $S_j(t;z_1,\dots,z_N;\mu)$) to $v_j(t)$.
\end{Proposition}

Next, it is proved in \cite[Theorem, 5.7]{FFTL}, that for any
collection of $\gg$-modules $M_i$ with highest weights $\l_i$,
$i=1,\dots,N$, the natural homomorphism
$$
\A_\mu(\ul{z})\to\End(M_1\otimes\dots\otimes M_N)
$$
factors through the algebra of functions on the subspace $\Op_{^L
G}(\PP^1)_{\ul{z};\pi(-\mu)}^{\ul{\l}}\subset\Op_{^L
G}(\PP^1)_{\ul{z};\pi(-\mu)}$ which consists of the opers with the
$1$-residue $\pi(-\l_i-\rho)$ at $z_i$. Moreover, for integral
dominant $\l_i$ the action of $\A_\mu(\ul{z})$ on the tensor
product of the finite-dimensional modules $V_{\l_i}$ (resp. the action of $\A(\ul{z})$ on $V_{\ul{\l}}^{sing}$) factors through
the algebra of functions on \emph{monodromy-free} opers from $\Op_{^L
G}(\PP^1)_{\ul{z};\pi(-\mu)}^{\ul{\l}}$ (resp. from $\Op_{^L
G}(\PP^1)_{\ul{z}}^{\ul{\l}}$), cf. \cite{FG}.

For every integral dominant weight $\l$, the set of monodromy-free
opers on the punctured disc $D_z^\times$ at $z$ with the regular
singularity with the residue $-\l-\rho$ at $z$ is defined by finitely
many polynomial relations. These polynomial relations are indexed by positive roots of $\g$, see \cite{FG}.

For any collection of integral dominant weights $\l_1,\dots,\l_N$
attached to the points $z_1,\dots,z_N$, we denote by
$P_{z_k;\alpha}^{\ul{z};\mu;\ul{\l}}$ the
polynomial in $v_{j,n}^{(i)}$ expressing the "no-monodromy" condition
at $z_k$ corresponding to the root $\alpha \in \Delta_+$. We have the following result:

\begin{Theorem}\label{theorem-spectra}\cite{spectra}
For any $N$-tuple of pairwise distinct complex numbers
$z_1,\dots,z_N\in\cc$ and any regular $\mu$ the subalgebra
$\A_\mu(z_1,\dots,z_N)\subset U(\gg)^{\otimes N}$ has a cyclic vector
in $V_{\ul{\l}}=V_{\l_1}\otimes\dots\otimes V_{\l_N}$.  The annihilator
of $V_{\ul{\l}}$ in $\A_\mu(z_1,\dots,z_N)$ is generated by the
no-monodromy conditions $P_{z_k;\alpha}^{\ul{z};\mu;\ul{\l}}$. In particular, the joint eigenvalues of
$\A_\mu(z_1,\dots,z_N)$ in $V_\l$ (without multiplicities) are in
one-to-one correspondence with the monodromy-free opers from $\Op_{^L
G}(\PP^1)_{\ul{z};\pi(-\mu)}^{\ul{\l}}$.
\end{Theorem}

\begin{Remark} This theorem can be regarded as an inhomogeneous version of the Bethe Ansatz conjecture. It turns to be much easier than the homogeneous one since the polynomials $P_{z_k;\alpha}^{\ul{z};\mu;\ul{\l}}$ form a regular sequence in the polynomial ring $\cc[\Op_{^L G}(\PP^1)_{\ul{z};\pi(-\mu)}^{\ul{\l}}]$ and hence one can check that the dimension of the quotient ring coincides with the dimension of $V_{\ul{\l}}$. There is no analogous argument for the homogeneous Gaudin model.
\end{Remark}

Denote by $P_{z_k;\alpha}^{\ul{z};0;\ul{\l}}$ the restrictions of the polynomials $P_{z_k;\alpha}^{\ul{z};f;\ul{\l}}$ to the subspace $\Op_{^L
G}(\PP^1)_{\ul{z}}\subset\Op_{^L
G}(\PP^1)_{\ul{z};\pi(-f)}^{\ul{\l}}$. The set of common zeroes of all $P_{z_k;\alpha}^{\ul{z};0;\ul{\l}}$  is the set of monodromy-free opers from $\Op_{^L
G}(\PP^1)_{\ul{z}}^{\ul{\l}}$.

\section{Main results}\label{sect-results}

\subsection{Formulation of the Main Theorem}

\begin{Theorem}\label{thm-main} The space of singular vectors in the tensor product of $\gg$-modules $V_{\ul{\l}}^{sing}$ is cyclic as
an $\A(\ul{z})$-module. The annihilator of $V_{\ul{\l}}^{sing}$ in $\A(\ul{z})$ is generated
by the no-monodromy conditions $P_{z_k;\alpha}^{\ul{z};0;\ul{\l}}$ on opers from $\Op_{^L G}(\PP^1)_{\ul{z}}^{\ul{\l}}$.
\end{Theorem}

This Theorem has an obvious corollary which is often referred to as Bethe ansatz conjecture (in the Feigin-Frenkel form):

\begin{Corollary}\label{cor-main} There is a bijection between the set of joint eigenvalues of $\A(\ul{z})$ (without multiplicities) on $V_{\ul{\l}}^{sing}$ and the set of
  monodromy-free opers from $\Op_{^L
G}(\PP^1)_{\ul{z}}^{\ul{\l}}$.
\end{Corollary}

Moreover, generically (e.g. for all real values of the parameters) there are no multiplicities:

\begin{Corollary}\label{cor-simple-spectrum}
For generic $z_1,\dots,z_N\in\cc$ the subalgebra
$\A(\ul{z})\subset U(\gg)^{\otimes N}$ has simple
spectrum in $V_{\ul{\l}}^{sing}=(V_{\l_1}\otimes\dots\otimes V_{\l_N})^{sing}$. Hence
the joint eigenvectors for higher Gaudin Hamiltonians in
$V_{\ul{\l}}^{sing}$ are in one-to-one correspondence with monodromy-free
opers from $\Op_{^L
G}(\PP^1)_{\ul{z}}^{\ul{\l}}$. All real $z_1,\dots,z_N$ are generic in this sense.
\end{Corollary}

\begin{proof}
It is proved in \cite{spectra} that for real values of the parameters $\ul{z}$ and $\mu$, the algebra $\A_\mu(\ul{z})$ acts on $V_{\ul{\l}}$ by Hermitian operators hence diagonalizable. On the other hand, by the Main Theorem, $\A(\ul{z})$ has a cyclic vector in $V_{\ul{\l}}^{sing}$. Hence for real $\ul{z}$ the algebra $\A(\ul{z})$ acts with simple spectrum on $V_{\ul{\l}}^{sing}$. The latter is a Zariski open condition on the parameters hence the assertion.
\end{proof}

\subsection{Proof of the Main
Theorem.}\label{subsect-proof}

Let $e,f,h$ be a principal $\fsl_2$-triple in $\fg$ such that $h\in\fh$ and $e\in\fn_+$. Then $f$ is a regular element of $\fg$. Consider the subalgebra $\A_{f}(\ul{z})\subset U(\fg)^{\otimes N}$. The operator $\ad \diag(h)$ defines a grading on $U(\fg)^{\otimes N}=\bigoplus\limits_{k\in\BZ} U(\fg)^{\otimes N}_k$, where $U(\fg)^{\otimes N}_k:=\{x\in U(\fg)^{\otimes N}\ |\ \diag(h)x-x\diag(h) = 2kx\}$. This grading induces an increasing filtration on $U(\fg)^{\otimes N}$ such that $U(\fg)^{(k)}:=\bigoplus\limits_{i\le k} U(\fg)^k$, $k\in\BZ$. Note that $\A_{f}(\ul{z})\subset U(\fg)^{(0)}$, hence we get a \emph{bounded} increasing filtration on $\A_{f}(\ul{z})$. The associated graded $\gr \A_{f}(\ul{z})$ is naturally a commutative subalgebra in $\gr U(\fg)^{\otimes N} = U(\fg)^{\otimes N}$ and is clearly the same as the limit $\lim\limits_{s\to0}\A_{sf}(\ul{z})$. According to Proposition~\ref{prop-lim} the $0$-th graded component of this subalgebra is $\A(\ul{z})\subset U(\fg)^{\otimes N}$.

Consider the action of $\A_{f}(\ul{z})$ on $V_{\ul{\l}}$. By \cite{spectra}, $V_{\ul{\l}}$ is cyclic as a $\A_{f}(\ul{z})$-module. The operator $\diag(h)$ defines a grading on $V_{\ul{\l}}$ which agrees with the grading of $U(\fg)^{\otimes N}$ defined above and with the decomposition into $\fg$-isotypic components with respect to the diagonal $\fg$-action $V_{\ul{\l}}=\bigoplus_{\nu}I_\nu$ (here $I_\nu=m_\nu V_\nu$ are isotypic components with respect to diagonal $\fg$). Clearly we have $V_\l^{\fn_+}=\bigoplus\limits_{\nu}I_\nu^{top}$ (here the superscript $top$ means the top degree part with respect to the $\diag(h)$-grading). Now we want to deduce the cyclicity of $I_\nu^{top}$ with respect to $\A(\ul{z})$ from the cyclicity of $V_{\ul{\l}}$ with respect to $\A_{f}(\ul{z})$. The main difficulty is that the decomposition $V_{\ul{\l}}=\bigoplus_{\nu}I_\nu$ does not agree with the action of $\A_{f}(\ul{z})$. However, the following holds (and it is sufficient for our purposes):

\begin{Lemma}\label{lem-isotypic} There is a decomposition $V_{\ul{\l}}=\bigoplus_{\nu} J_\nu$ such that it is preserved by $\A_{f}(\ul{z})$ and $\gr J_\nu=I_\nu$.
\end{Lemma}

\begin{proof} The isotypic components $I_\nu$ are the joint eigenspaces for the center $Z$ of the diagonal $U(\fg)$. Since the direct sum is finite $V_{\ul{\l}}=\bigoplus_{\nu} I_\nu$, it is in fact the decomposition into the eigenspaces for a single (sufficiently generic) element $C\in Z$. Since $Z\subset \A(\ul{z})\subset\gr\A_{f}(\ul{z})$, for any $C\in Z$ there exists $\tilde{C}\in\A_{f}(\ul{z})$ such that $\tilde{C}=C+N$ where $\deg N <0$. Denote by $\pi_{\ul{\l}}$ the homomorphism $U(\fg)^{\otimes N}\to\End(V_{\ul{\l}})$ and consider $\pi_{\ul{\l}}(\tilde{C})$, the image of $\tilde{C}$ in $\End(V_{\ul{\l}})$. We have the Jordan decomposition $\pi_{\ul{\l}}(\tilde{C})=\pi_{\ul{\l}}(\tilde{C})_s+\pi_{\ul{\l}}(\tilde{C})_n$ where both $\pi_{\ul{\l}}(\tilde{C})_s$ and $\pi_{\ul{\l}}(\tilde{C})_n$ are polynomials of $\pi_{\ul{\l}}(\tilde{C})$, hence lie in $\pi_{\ul{\l}}(\A_{f}(\ul{z}))$. Since $\pi_{\ul{\l}}(C)$ is a semisimple operator and $\pi_\l(\tilde{C})_n$ is a nilpotent operator expressed as a polynomial of $\pi_{\ul{\l}}(C+N)$, we have $\deg \pi_{\ul{\l}}(\tilde{C})_n<0$. Hence $\pi_{\ul{\l}}(\tilde{C})_s = \pi_{\ul{\l}}(C)$ modulo lower degree. Hence the projectors to the eigenspaces of $\pi_{\ul{\l}}(\tilde{C})_s$ are the the projectors to the eigenspaces of $C$ modulo lower degree. Thus the desired decomposition of $V_{\ul{\l}}$ is the decomposition into the eigenspaces with respect to $\pi_{\ul{\l}}(\tilde{C})_s$ for some generic $C\in Z$.
\end{proof}

By \cite{spectra}, $V_{\ul{\l}}$ is cyclic as a $\A_{f}(\ul{z})$-module. Hence each $J_\nu$ is cyclic with respect to $\A_{f}(\ul{z})$. In particular, the quotient $J_\nu^{(top)}/J_\nu^{(top-1)}$ is cyclic with respect to
$\A_{f}(\ul{z})/\A_{f}(\ul{z})_{(-1)}=\A(\ul{z})$. On the other hand, according to Lemma~\ref{lem-isotypic} $J_\nu^{(top)}/J_\nu^{(top-1)}$ is isomorphic to $I_\nu^{top}=m_\nu V_\nu^{\fn_+}$ as $\A(\ul{z})$-module. Hence the space of highest vectors of each $\diag(\fg)$-isotypic component in $V_{\ul{\l}}$ is cyclic as $\A(\ul{z})$-module, and hence the whole space $V_{\ul{\l}}^{sing}$ is a cyclic $\A(\ul{z})$-module. The first assertion of the Theorem is proved.

To prove the second assertion of the Theorem, we note that the filtration on the coordinate ring $\cc[\Op_{^L G}(\PP^1)_{\ul{z};\pi(-f)}]$ defined by the isomorphism with $\A_{f}(\ul{z})$ has, according to Proposition~\ref{multi-point-isomorphism}, the following meaning in terms of opers: the degrees of the generators are $\deg u_{j,n}=-d_j+n$, $\deg v_{j,n}^{(i)}=0$. Hence the $0$-th component of $\gr\cc[\Op_{^L G}(\PP^1)_{\ul{z};\pi(-f)}]$ with respect to this filtration is $\cc[\Op_{^L G}(\PP^1)_{\ul{z}}]$. The same holds for the quotients by the ``no-monodromy'' relations: the $0$-th component of $\gr\cc[\Op_{^L G}(\PP^1)_{\ul{z};\pi(-f)}^{\ul{\l}}]/(P_{z_k;\alpha}^{\ul{z};f;\ul{\l}})$ is $\cc[\Op_{^L G}(\PP^1)_{\ul{z}}^{\ul{\l}}]/(P_{z_k;\alpha}^{\ul{z};0;\ul{\l}})$. The natural homomorphism from $\cc[\Op_{^L G}(\PP^1)_{\ul{z}}^{\ul{\l}}]/(P_{z_k;\alpha}^{\ul{z};0;\ul{\l}})$ to the image of $\A(\ul{z})$ in $\End(V_{\ul{\l}}^{sing})$ is surjective by definition. For the second assertion of the Theorem, it suffices to show the injectivity of this homomorphism.

\begin{Lemma}\label{lem-isotypic2} Let $L\subset J_\nu$ be any lifting of $J_\nu^{(top)}/J_\nu^{(top-1)}$ to $J_\nu$. Then we have $J_\nu=L+\A_{f}(\ul{z})_{(-1)}L$.
\end{Lemma}
\begin{proof} It is sufficient to check that $I_\nu=I_\nu^{top}+\gr\A_{f}(\ul{z})_{(-1)}I_\nu^{top}$, i.e. that the isotypic component of $V_\nu$ is generated from its highest weight space by the negatively graded part of $\gr\A_{f}(\ul{z})$. We have $\gr\A_{f}(\ul{z})\subset\A_{0}(\ul{z})\cdot\diag(\A_{f})$, where $\A_{0}(\ul{z})$ has degree $0$ and the generators of $\A_{f}$ are either central or have negative degree, and by the results of \cite{spectra} each $V_\nu$ is generated by $\A_{f}$ from its highest vector.
\end{proof}

By Lemma~\ref{lem-isotypic2}, we have $\dim \pi_{\ul{\l}}(\A_f(\ul{z}))/\pi_{\ul{\l}}(\A_f(\ul{z})_{-1})\le\dim V_{\ul{\l}}^{sing}$. On the other hand, we have $\pi_{\ul{\l}}(\A_f(\ul{z}))=\cc[\Op_{^L G}(\PP^1)_{\ul{z};\pi(-f)}^{\ul{\l}}]/(P_{z_k;\alpha}^{\ul{z};f;\ul{\l}})$ by \cite{spectra}. Hence $\dim \cc[\Op_{^L G}(\PP^1)_{\ul{z}}^{\ul{\l}}]/(P_{z_k;\alpha}^{\ul{z};0;\ul{\l}})=\dim(\cc[\Op_{^L G}(\PP^1)_{\ul{z}}^{\ul{\l}}]/(P_{z_k;\alpha}^{\ul{z};f;\ul{\l}})+\cc[\Op_{^L G}(\PP^1)_{\ul{z}}^{\ul{\l}}]_{-1})\le\dim V_{\ul{\l}}^{sing}$. On the other hand, by the first assertion of the Theorem, the dimension of the image of $\A(\ul{z})$ in $\End(V_{\ul{\l}}^{sing})$ is equal to $\dim V_{\ul{\l}}^{sing}$ and by the general results of \cite{FG} the homomorphism $\cc[\Op_{^L G}(\PP^1)_{\ul{z}}]=\A(\ul{z})\to\pi_{\ul{\l}}(\A(\ul{z}))$ factors through $\cc[\Op_{^L G}(\PP^1)_{\ul{z}}^{\ul{\l}}]/(P_{z_k;\alpha}^{\ul{z};0;\ul{\l}})$. Thus $\cc[\Op_{^L G}(\PP^1)_{\ul{z}}^{\ul{\l}}]/(P_{z_k;\alpha}^{\ul{z};0;\ul{\l}})$ has the same dimension as the image of $\A(\ul{z})$ in $\End(V_{\ul{\l}}^{sing})$ and we are done.
$\qquad\Box$

\begin{Remark} The same argument proves the cyclicity property for $\A_\mu(\ul{z})$ with any semisimple $\mu$. Namely, one can show that there is a cyclic vector for $\A_\mu(\ul{z})$ in the space of highest vectors with respect to the (reductive) Lie algebra $\mathfrak{z}_\fg(\mu)$, the centralizer of $\mu$ in $\fg$. The only difference is that one has to take a principal nilpotent element $f\in\mathfrak{z}_\fg(\mu)$ (then $\A_\mu+f(\ul{z})$ is an inhomogeneous Gaudin algebra with regular parameter) and consider the grading coming from a principal $\fsl_2$-triple in $\mathfrak{z}_\fg(\mu)$, not in $\fg$. We will discuss this in detail in the forthcoming joint paper with Joel Kamnitzer.

On the other hand, we do not know an oper description of the spectrum in this generality.
\end{Remark}

\section{Limit Gaudin subalgebras.}\label{sect-cactus}

According to Proposition~\ref{generators2} the commutative subalgebras $\A(\ul{z})\subset U(\fg)^{\otimes N}$ form a \emph{flat} family parameterized by the configuration space $$M_{0,N+1}=\{z_1,\ldots,z_N\ |\ z_i\ne z_j\}/\{z\mapsto az+b\}$$ of stable rational curves with marked points. This means that for any positive integer $M$ the dimension of the intersection of $\A(\ul{z})$ with the $M$-th filtered component $\rm{PBW}^{(M)}U(\fg)^{\otimes N}$ with respect to the PBW filtration of the universal enveloping algebra $U(\fg)^{\otimes N}$ does not depend on $\ul{z}\in M_{0,N+1}$. Denote this dimension by $d(M,N)$. Then there is a regular map from $M_{0,N+1}$ to the product of the Grassmannians $\prod\limits_{M'\le M}{\rm Gr}(d(M',N),{\rm PBW}^{(M')}U(\fg)^{\otimes N})$ taking $\ul{z}\in M_{0,N+1}$ to $\prod\limits_{M'\le M}\A(\ul{z})\cap {\rm PBW}^{(M')}U(\fg)^{\otimes N}$. Let $Z_M$ be the closure of the image of this map. Then there are surjective restriction maps $r_{MM'}:Z_M\to Z_{M'}$ for any $M'<M$. The inverse limit $Z=\lim\limits_{\leftarrow} Z_M$ is well-defined as a pro-algebraic scheme. The restriction of the tautological vector bundle on the Grassmannian gives a sheaf $\A$ of commutative algebras on $Z$. In \cite{cactus}, we show that $Z$ is biregular to the Deligne-Mumford compactification $\overline{M_{0,N+1}}$ of the moduli space of stable rational curves with $N+1$ marked points. The limit Gaudin subalgebras (i.e. the fibers of $\A$ over boundary points of $\overline{M_{0,N+1}}$) are described in \cite{cactus} as follows.

For any partition of the set $\{1,2,\ldots,N\}=M_1\cup\ldots\cup M_k$, define the homomorphism $$D_{M_1,\ldots,M_k}:U(\fg)^{\otimes k}\hookrightarrow
U(\fg)^{\otimes n},$$ taking $x^{(i)}\in U(\fg)^{\otimes k}$, for $x\in\fg,\ i=1,\ldots,k$, to $\sum\limits_{j\in M_i}x^{(j)}$.

For any subset $M=\{j_1,\ldots,j_m\}\subset\{1,2,\ldots,N\}$, with $j_1<\ldots<j_m$, let $I_M:U(\fg)^{\otimes m}\hookrightarrow
U(\fg)^{\otimes n}$ be the embedding of the tensor product of the copies of $U(\fg)$ indexed by $M$, i.e. $I_M(x^{(i)}):=x^{(j_i)}\in U(\fg)^{\otimes N}$ for any $x\in\fg,\ i=1,\ldots,m$. Clearly, all these homomorphisms are $\fg$-equivariant and every element in the image of $D_{M_1,\ldots,M_k}$ commutes with every element of $I_{M_i}([U(\fg)^{\otimes m_i}]^\fg)$ for $i=1,\ldots,k$. This gives us the following ``substitution'' homomorphism defining an operad structure on the spaces $[U(\fg)^{\otimes n}]^\fg$
\begin{equation*}
\gamma_{k;M_1,\ldots,M_k}=D_{M_1,\ldots,M_k}\otimes\bigotimes\limits_{i=1}^k I_{M_i}:(U(\fg)^{\otimes k})^\fg\otimes\bigotimes\limits_{i=1}^k(U(\fg)^{\otimes m_i})^\fg\to (U(\fg)^{\otimes N})^\fg.
\end{equation*}

\begin{Proposition}\cite{cactus} All subalgebras corresponding to boundary points of $\overline{M_{0,N+1}}$ are images of products of some $\A(\ul{z})$ under compositions of the homomorphisms $\gamma$.
\end{Proposition}

This implies the following

\begin{Proposition}\label{prop-limit-subalg}
For any $\ul{z}\in \ol{M_{0,N+1}}$ the subalgebra
$\A(\ul{z})\subset U(\gg)^{\otimes N}$ has a cyclic vector in $V_{\ul{\l}}^{sing}=(V_{\l_1}\otimes\dots\otimes V_{\l_N})^{sing}$. For any $\ul{z}\in \ol{M_{0,n+1}(\rr)}$ the subalgebra
$\A(\ul{z})\subset U(\gg)^{\otimes N}$ has simple spectrum in $V_{\ul{\l}}^{sing}=(V_{\l_1}\otimes\dots\otimes V_{\l_N})^{sing}$.
\end{Proposition}

\begin{proof} The proof repeats that of Theorem~3.18 from \cite{cactus}. We reproduce it here for completeness.

To prove the first assertion, we proceed by induction on $N$. Suppose that $\ul{z}=\gamma_{k;M_1,\ldots,M_k}(\ul{w};\ul{u_1},\ldots,\ul{u_k})$. Then the corresponding subalgebra $\A(\ul{z})$ is generated by $I_{M_i}(\A(\ul{u_i}))$ and $D_{M_1,\ldots,M_k}(\A(\ul{w}))$. Let $V_{\ul{\l}}=\bigoplus W_{\ul{\nu}}\otimes V_{\ul{\nu}}$ be the decomposition of $V_{\ul{\l}}$ into the sum of isotypic component with respect to $D_{M_1,\ldots,M_k}(\fg^{\oplus k})$ with $V_{\ul{\nu}}$ being the irreducible representation of $\fg^{\oplus k}$ with the highest weight $\ul{\nu}=(\nu_1,\ldots,\nu_k)$ and $W_{\ul{\nu}}:=\Hom_{\fg^{\oplus k}}(V_{\ul{\nu}},V_{\ul{\l}})$ being the multiplicity space. By induction hypothesis, the multiplicity spaces $W_{\ul{\nu}}$ are cyclic $\bigotimes\limits_{i=1}^k I_{M_i}(\A(\ul{u_i}))$-modules. On the other hand, the space of singular vectors of each $V_{\ul{\nu}}$ is a cyclic $D_{M_1,\ldots,M_k}(\A(\ul{w}))$-module. Hence the entire module $V_{\ul{\l}}^{sing}$ is cyclic with respect to $\A(\ul{z})$.

For real $\ul{z}$ the generators of $\A(\ul{z})$ act by Hermitian operators in any $V_{\ul{\l}}$, and hence are diagonalizable. Since there is a cyclic vector for the action of $\A(\ul{z})$ in the space $V_{\ul{\l}}^{sing}$ the joint eigenvalues of the generators on different eigenvectors are different.
\end{proof}

\begin{Corollary}\label{cor-covering} For any collection $\ul{\l}$ of dominant integral weights, the spectra of the algebras $\A(\ul{z})$ with \emph{real} $\ul{z}$ in the space $V_{\ul{\l}}^{sing}$ form an unbranched covering of $\overline{M_{0,N+1}(\BR)}$.
\end{Corollary}

\end{document}